\documentclass[11pt]{article}%
\usepackage{amsmath}
\usepackage{amsfonts}
\usepackage{amssymb}
\usepackage{graphicx}%
\newtheorem{theorem}{Theorem}

\newtheorem{definition}[theorem]{Definition}

\newtheorem{lemma}[theorem]{Lemma}

\newtheorem{proposition}[theorem]{Proposition}

\newenvironment{proof}[1][Proof]{\noindent\textbf{#1.} }{\ \rule{0.5em}{0.5em}}
\begin{document}

\title{Some spherical function values for two-row tableaux and Young subgroups with
three factors}
\author{Charles F. Dunkl\thanks{email: cfd5z@virginia.edu}\\Department of Mathematics\\University of Virginia\\Charlottesville, VA 22904-4137}
\maketitle

\begin{abstract}
A Young subgroup of the symmetric group $\mathcal{S}_{N}$ with three factors,
is realized as the stabilizer $G_{n}$ of a monomial $x^{\lambda}$ (
$=x_{1}^{\lambda_{1}}x_{2}^{\lambda_{2}}\cdots x_{N}^{\lambda_{N}}$) with
$\lambda=\left(  d_{1}^{n_{1}},d_{2}^{n_{2}},d_{3}^{n_{3}}\right)  $ (meaning
$d_{j}$ is repeated $n_{j}$ times, $1\leq j\leq3$), thus is isomorphic to the
direct product $\mathcal{S}_{n_{1}}\times\mathcal{S}_{n_{2}}\times
\mathcal{S}_{n_{3}}$. The orbit of $x^{\lambda}$ under the action of
$\mathcal{S}_{N}$ (by permutation of coordinates) spans a module $V_{\lambda}%
$, the representation induced from the identity representation of $G_{n}$. The
space $V_{\lambda}$ decomposes into a direct sum of irreducible $\mathcal{S}%
_{N}$-modules. The spherical function is defined for each of these, it is the
character of the module averaged over the group $G_{n}$. This paper concerns
the value of certain spherical functions evaluated at a cycle which has no
more than one entry in each of the three intervals $I_{j}=\left\{
i:\lambda_{i}=d_{j}\right\}  ,1\leq j\leq3$. These values appear in the study
of eigenvalues of the Heckman-Polychronakos operators in the paper by V. Gorin
and the author (arXiv:2412:01938v1). The present paper determines the
spherical function values for $\mathcal{S}_{N}$-modules $V$ of two-row tableau
type, corresponding to Young tableaux of shape $\left[  N-k,k\right]  $. The
method is based on analyzing the effect of a cycle on $G_{n}$-invariant
elements of $V$. These are constructed in terms of Hahn polynomials in two variables.

\end{abstract}

\section{Introduction}

Spherical functions arise when an irreducible representation of a group
contains the identity representation of a subgroup. This paper concerns the
symmetric group and subgroups of Young type. Such groups are defined as
stabilizer groups of particular monomials in the context of the symmetric
group acting on polynomials by permutation of variables. Specifically we study
the Young subgroup of $\mathcal{S}_{N}$ leaving each of three subintervals
$I_{1}=\left[  1,n_{1}\right]  $, $I_{2}=\left[  n_{1}+1,n_{1}+n_{2}\right]
$, $I_{3}=\left[  n_{1}+n_{2}+1,N\right]  $ setwise invariant , where
$N=n_{1}+n_{2}+n_{3}$ and $\mathcal{S}_{N}$ is the symmetric group of
permutations of $\left[  1,N\right]  =\left\{  1,2,\ldots,N\right\}  $. The
goal is to evaluate the spherical function for the isotype described by
two-row tableaux at cycles which have at most one entry in each of the
subintervals. This problem comes from a paper by Gorin and the author
\cite{DG} which analyzed the eigenvalues of certain difference-differential operator.

The basic technique is to specify a submodule of polynomials realizing the
isotype $\left[  N-k,k\right]  $ with $2k\leq N$, describe the polynomials
invariant under the Young subgroup, act on each of these by the cycle of
interest, and then project onto the space of invariants. The spherical
function is then computed from this data. In the present situation the
invariant polynomials are expressed with the aid of certain Hahn polynomials
in two variables.

We begin with a brief sketch of the background from \cite{DG}. The commutative
family of Heckman-Polychronakos operators is the set $\mathcal{P}_{k}%
:=\sum_{i=1}^{N}\left(  x_{i}\mathcal{D}_{i}\right)  ^{k}$ ($k\geq1$) in terms
of Dunkl operators $\mathcal{D}_{i}f\left(  x\right)  :=\frac{\partial
}{\partial x_{i}}f\left(  x\right)  +\kappa\sum_{j=1,j\neq i}^{N}%
\frac{f\left(  x\right)  -f\left(  x\left(  i,j\right)  \right)  }{x_{i}%
-x_{j}}$; $x\left(  i,j\right)  $ denotes $x$ with $x_{i}$ and $x_{j}$
interchanged, and $\kappa$ is a fixed parameter, often satisfying
$\kappa>-\frac{1}{N}$ (see Heckman \cite{He}, Polychronakos \cite{Po}). For
$\alpha\in\mathbb{Z}_{+}^{N}$, let $x^{\alpha}:=\prod\limits_{i=1}^{N}%
x_{i}^{\alpha_{i}}$. Suppose $\lambda_{1}\geq\lambda_{2}\geq\ldots\geq
\lambda_{N}\geq0$ then set $V_{\lambda}=\mathrm{span}_{\mathbb{F}}\left\{
x^{\beta}:\beta=w\lambda,w\in\mathcal{S}_{N}\right\}  $, that is, $\beta$
ranges over the permutations of $\lambda$, and $\mathbb{F}$ is an extension
field of $\mathbb{R}$ containing at least $\kappa$. The space $V_{\lambda}$ is
invariant under the action of $\mathcal{S}_{N}$. The eigenvalue analysis of
$\mathcal{P}_{k}$ is derived from the restriction of $\mathcal{P}%
_{k}V_{\lambda}$ to $V_{\lambda}$ (there is a triangularity based on the
dominance order of partitions). Let $\lambda=\left(  d_{1}^{n_{1}}%
,d_{2}^{n_{2}},d_{3}^{n_{3}}\right)  $ (that is, $d_{j}$ is repeated $n_{j}$
times, $1\leq j\leq3$), with $d_{1}>d_{2}>d_{3}\geq0$. Let $G_{\mathbf{n}}$
denote the stabilizer group of $x^{\lambda}$, so that $G_{\mathbf{n}}%
\cong\mathcal{S}_{n_{1}}\times\mathcal{S}_{n_{2}}\times\mathcal{S}_{n_{3}}$.
The representation of $\mathcal{S}_{N}$ realized on $V_{\lambda}$ is the
induced representation $\mathrm{ind}_{G_{\mathbf{n}}}^{\mathcal{S}_{N}}$. The
space $V_{\lambda}$ can be decomposed into a direct sum of $\mathcal{S}_{N}%
$-invariant subspaces of various isotypes, which may appear as several copies.
The number of copies (the multiplicity) of a particular isotype $\tau$ is
called a \emph{Kostka} number (see Macdonald \cite{MacD}). That is,
$V_{\lambda}=\sum_{\tau}\oplus V_{\lambda;\tau}.$ Because $\mathcal{P}_{k}$
commutes with the group action the restriction of $\mathcal{P}_{k}%
V_{\lambda;\tau}$ to $V_{\lambda}$ is contained in $V_{\lambda;\tau}$. If the
multiplicity of the isotype $\tau$ in $V_{\lambda}$ is greater than one then
the eigenvalues of $\mathcal{P}_{k}$ realized on $V_{\lambda;\tau}$ are
generally not rational in the parameters, but the sum of all the eigenvalues
(for any fixed $k$) can be explicitly found, in terms of the character of
$\tau$. In general this may not have a simple explicit form . A closed form
was found for hook isotypes, labeled by partitions of the form $\left[
N-b,1^{b}\right]  $ (by the author \cite{Dhk}, for the more general
$G_{\mathbf{n}}\cong\mathcal{S}_{n_{1}}\times\mathcal{S}_{n_{2\cdots}}%
\times\cdots\times\mathcal{S}_{n_{p}}$). The formula is based on considering
cycles corresponding to subsets $\mathcal{A=}\left\{  a_{1},\ldots,a_{\ell
}\right\}  $ of $\left\{  1,2,3\right\}  $, which are of length $\ell$ with
exactly one entry from each interval $I_{a_{j}}$. Any such cycle can be used
and the order of $a_{1},\cdots,a_{\ell}$ is immaterial. The degrees
$d_{1},d_{2},d_{3}$ enter the formula in a shifted way:%
\[
\widetilde{d}_{1}:=d_{1}+\kappa\left(  n_{2}+n_{3}\right)  ,\widetilde{d}%
_{2}:=d_{2}+\kappa n_{3},\widetilde{d}_{3}:=d_{3}.
\]
Let $h_{m}^{\mathcal{A}}:=h_{m}\left(  \widetilde{d}_{a_{1}},\widetilde
{d}_{a_{2}},\ldots,\widetilde{d}_{a_{\ell}}\right)  $, the complete symmetric
polynomial of degree $m$ (the generating function is $\sum_{k\geq0}%
h_{k}\left(  c_{1},c_{2},\ldots,c_{q}\right)  t^{k}=\prod_{i=1}^{q}\left(
1-c_{i}t\right)  ^{-1}$, see \cite[p.21]{MacD}). Denote the character of the
representation $\tau$ of $\mathcal{S}_{N}$ by $\chi^{\tau}\left(  w\right)  $
then the spherical function%
\[
\Phi^{\tau}\left(  g_{\mathcal{A}}\right)  :=\frac{1}{\#G_{\mathbf{n}}}%
\sum_{h\in G_{\mathbf{n}}}\chi^{\tau}\left(  g_{\mathcal{A}}~h\right)  ,
\]
where $g_{\mathcal{A}}$ is an $\ell$-cycle labeled by $\mathcal{A}$ as above,
and $\#G_{\mathbf{n}}=\prod_{i=1}^{3}n_{i}!$. In \cite{DG} the spherical
function $\Phi^{\tau}$ is denoted by $\chi^{\tau}\left[  \mathcal{A}%
;\mathbf{n}\right]  $, and called an "averaged character."

Now suppose the multiplicity of $\tau$ in $V_{\lambda}$ is $\mu$ then there
are $\mu\dim\tau$ eigenfunctions and eigenvalues of $\mathcal{P}_{k}$, and the
sum of all these eigenvalues is (\cite[Thm. 5.4]{DG})%
\[
\dim\tau\sum_{\ell=1}^{\min\left(  k+1,3\right)  }\left(  -\kappa\right)
^{\ell-1}\sum_{\mathcal{A}\subset\left\{  1,2,3\right\}  ,\#\mathcal{A}=\ell
}\Phi^{\tau}\left(  g_{\mathcal{A}}\right)  ~h_{k+1-\ell}^{\mathcal{A}}%
\prod_{i\in\mathcal{A}}n_{i}!.
\]
Here is an outline of the paper. Section \ref{SphF} reviews some general
results about spherical functions and the formula proven in \cite{Dhk} which
is the basic tool for the computations. The derivation starts with the
construction of an irreducible module of polynomials of isotype $\tau$ (for
practical reasons we choose such a module of minimum polynomial degree). The
invariant polynomials in $V$ are described in terms of elementary symmetric
polynomials. The dimension of the subspace of invariants is found in terms of
the parameters $n_{1},n_{2},n_{3}$ (by Frobenius reciprocity the dimension is
the same as $\mu$, the multiplicity of $\left[  N-k,k\right]  $ in
$V_{\lambda})$. In Section \ref{HahnP2} the two-variable Hahn polynomials are
defined and a basis for the invariants is constructed. Section \ref{twocyc}
determines the spherical functions for the $2$-cycles. Lastly Section
\ref{threecyc} produces the spherical function for $3$-cycles and also has a
discussion (Subsection \ref{mult1}) about some specific examples, especially
those with multiplicity one.

\section{\label{SphF}Spherical functions and invariants}

\begin{definition}
The action of the symmetric group $\mathcal{S}_{N}$ on polynomials $P\left(
x\right)  $ is given by $wP\left(  x\right)  =P\left(  xw\right)  $ and
$\left(  xw\right)  _{i}=x_{w\left(  i\right)  }$, $w\in\mathcal{S}_{N},1\leq
i\leq N$.
\end{definition}

Note $\left(  x\left(  vw\right)  \right)  _{i}=\left(  xv\right)  _{w\left(
i\right)  }=x_{v\left(  w\left(  i\right)  \right)  }=x_{vw\left(  i\right)
}$, $vwP\left(  x\right)  =\left(  wP\right)  \left(  xv\right)  =P\left(
xvw\right)  $. The projection onto $G_{\mathbf{n}}$-invariant polynomials is
given by%
\[
\rho P\left(  x\right)  :=\dfrac{1}{\#G_{\mathbf{n}}}%
{\displaystyle\sum\limits_{h\in G_{\mathbf{n}}}}
P\left(  xh\right)  .
\]
Let $\lambda=\left(  d_{1}^{n_{1}},d_{2}^{n_{2}},d_{3}^{n_{3}}\right)  $ (that
is, $d_{j}$ is repeated $n_{j}$ times, $1\leq j\leq3$), with $d_{1}%
>d_{2}>d_{3}\geq0$. Let $G_{\mathbf{n}}$ denote the stabilizer group of
$x^{\lambda}$, so that $G_{\mathbf{n}}\cong\mathcal{S}_{n_{1}}\times
\mathcal{S}_{n_{2}}\times\mathcal{S}_{n_{3}}$. Suppose that $M_{\tau}$ is an
$\mathcal{S}_{N}$-module of isotype $\tau$ and that $\left\{  \psi_{j}:1\leq
j\leq\mu\right\}  $ is a basis for the $G_{\mathbf{n}}$-invariants.

\begin{proposition}
\label{Bsum}\cite[Cor. 2]{Dhk} Suppose $g\in\mathcal{S}_{N}$ and $\rho
g\xi_{i}=\sum_{j=1}^{\mu}B_{ji}\left(  g\right)  \xi_{j}$ ($1\leq i,j\leq\mu$)
then $\Phi^{\tau}\left(  g\right)  =\mathrm{tr}\left(  B\left(  g\right)
\right)  $.
\end{proposition}

The key fact is that $\rho g\xi_{i}$ is itself an invariant and thus has a
unique expansion in the basis $\left\{  \psi_{j}:1\leq j\leq\mu\right\}  $.
The approach used in the sequel is to determine the action of $\rho g$ on each
basis element for the cycles described above.

For $1\leq k\leq\frac{N}{2}$ let $E\subset\left\{  1,2,\ldots,N\right\}  $
with $\#E=k$ and let $m_{E}:=\prod\limits_{i\in E}x_{i}$ and $V_{k}:=\left\{
p=\sum\limits_{\#E=k}c_{E}m_{E}:\sum\limits_{j=1}^{N}\frac{\partial}{\partial
x_{i}}p=0\right\}  $. Then $V_{k}$ is of isotype $\left[  N-k,k\right]  $ (an
irreducible $\mathcal{S}_{N}$-module of degree $\binom{N}{k}-\binom{N}{k-1}$ ).

To clearly display the action of $G_{\mathbf{n}}$ we introduce a modified
coordinate system. Replace
\[
\left(  x_{1},x_{2},\ldots,x_{N}\right)  \symbol{126}\left(  x_{1}^{\left(
1\right)  },\ldots,x_{n_{1}}^{\left(  1\right)  },x_{1}^{\left(  2\right)
},\ldots,x_{n_{2}}^{\left(  2\right)  },x_{1}^{\left(  3\right)  }%
,\ldots,x_{n_{p}}^{\left(  3\right)  }\right)  ,
\]
that is, $x_{i}^{\left(  j\right)  }$ stands for $x_{s}$ with $s=\sum
_{i=1}^{j-1}n_{i}+i$. We use $x_{\ast}^{\left(  i\right)  },x_{>}^{\left(
i\right)  }$ to denote a generic $x_{j}^{\left(  i\right)  }$ with $1\leq
j\leq n_{i}$, respectively $2\leq j\leq n_{i}$. In the sequel $g_{\ell}$
denotes the cycle $\left(  x_{1}^{\left(  1\right)  },x_{1}^{\left(  2\right)
},\ldots,x_{1}^{\left(  \ell\right)  }\right)  $ (with $2\leq\ell\leq3$). Let
$e_{i}\left(  x_{\ast}^{\left(  j\right)  }\right)  ,e_{i}\left(
x_{>}^{\left(  j\right)  }\right)  $ be defined by $\prod\limits_{i=1}^{n_{j}%
}\left(  1+tx_{i}^{\left(  j\right)  }\right)  =\sum\limits_{i=0}^{n_{i}}%
t^{i}e_{i}\left(  x_{\ast}^{\left(  j\right)  }\right)  $, respectively
$\prod\limits_{i=2}^{n_{j}}\left(  1+tx_{i}^{\left(  j\right)  }\right)
=\sum\limits_{i=0}^{n_{i}-1}t^{i}e_{i}\left(  x_{>}^{\left(  j\right)
}\right)  $ (elementary symmetric functions).

\begin{lemma}
\label{symme}$\rho\left(  x_{1}^{\left(  j\right)  }e_{i-1}\left(
x_{>}^{\left(  j\right)  }\right)  \right)  =\dfrac{i}{n_{j}}e_{i}\left(
x_{\ast}^{\left(  j\right)  }\right)  $ and $\rho\left(  e_{i}\left(
x_{>}^{\left(  j\right)  }\right)  \right)  =\dfrac{n_{j}-i}{n_{j}}%
e_{i}\left(  x_{\ast}^{\left(  j\right)  }\right)  $.
\end{lemma}

\begin{proof}
Let $p=x_{s_{1}}^{\left(  j\right)  }x_{s_{2}}^{\left(  j\right)  }\cdots
x_{s_{i}}^{\left(  j\right)  }$ (with $s_{1}<\ldots<s_{i}$) then $\rho
p=\binom{n_{j}}{i}^{-1}e_{i}\left(  x_{\ast}^{\left(  j\right)  }\right)  $,
because $e_{i}\left(  x_{\ast}^{\left(  j\right)  }\right)  $ is the sum of
$\binom{n_{j}}{i}$ monomials. There are $\binom{n_{j}-1}{i-1}$ monomials in
$x_{1}^{\left(  j\right)  }e_{i-1}\left(  x_{>}^{\left(  j\right)  }\right)  $
and $\binom{n_{j}-1}{i-1}/\binom{n_{j}}{i}=\dfrac{i}{n_{j}}$. There are
$\binom{n_{j}-1}{i}$ monomials in $e_{i}\left(  x_{>}^{\left(  j\right)
}\right)  $ and $\binom{n_{j}-1}{i}/\binom{n_{j}}{i}=\dfrac{n_{j}-i}{n_{j}}$.
\end{proof}

\begin{proposition}
\label{diffceqn}\cite[Prop. 2.1]{DPJ} A polynomial $f\left(  u,v\right)  $
satisfies%
\[
\left(  u-n_{1}\right)  f\left(  u+1,v\right)  +\left(  v-n_{2}\right)
f\left(  u,v+1\right)  -\left(  n_{3}-k+1+u+v\right)  f\left(  u,v\right)  =0
\]
for $0\leq u\leq n_{1},~0\leq v\leq n_{2},~k-n_{3}\leq u+v\leq k$ , if and
only if%
\[
\sum_{i=1}^{N}\frac{\partial}{\partial x_{i}}\sum_{u,v,u+v\leq k}f\left(
u,v\right)  e_{u}\left(  x_{\ast}^{\left(  1\right)  }\right)  e_{v}\left(
x_{\ast}^{\left(  2\right)  }\right)  e_{k-u-v}\left(  x_{\ast}^{\left(
3\right)  }\right)  =0,
\]
that is, the inner sum is an element of $V_{k}$.
\end{proposition}

The formula describes the space of $G_{\mathbf{n}}$-invariant polynomials of
isotype $\left[  N-k,k\right]  $.

\section{\label{HahnP2}Hahn polynomials in two variables}

One convenient orthogonal basis for $V_{k}$ is defined in terms of Hahn
polynomials (see \cite{DPJ} ) (the Pochhammer symbol is $\left(
\alpha\right)  _{j}=\prod_{i=1}^{j}\left(  \alpha+i-1\right)  $)%
\[
E_{m}\left(  \alpha,\beta,\gamma,t\right)  :=\sum_{i=0}^{m}\left(  -1\right)
^{i}\dbinom{m}{i}\left(  \beta-m+1\right)  _{i}\left(  \alpha-m+1\right)
_{m-i}\left(  -t\right)  _{i}\left(  t-\gamma\right)  _{m-i}%
\]
then the basis element $\psi_{m}$ is given by (in two parts for later
convenience)%
\begin{align}
\widetilde{\psi}_{m}^{\left(  1\right)  }\left(  t\right)   &  :=E_{k-m}%
\left(  n_{3},n_{1}+n_{2}-2m,k-m,k-t\right) \label{psim1}\\
&  =\sum_{j=0}^{k-m}\frac{\left(  m-k\right)  _{j}}{j!}\left(  n_{1}%
+n_{2}-k-m+1\right)  _{j}\left(  n_{3}-k+m+1\right)  _{k-m-j}\nonumber\\
&  \times\left(  t-k\right)  _{j}\left(  m-t\right)  _{k-m-j}\nonumber
\end{align}%
\begin{align}
\widetilde{\psi}_{m}^{\left(  2\right)  }\left(  u,v\right)   &
:=E_{m}\left(  n_{2},n_{1},u+v,v\right) \label{psim2}\\
&  =\sum_{i=0}^{m}\frac{\left(  -m\right)  _{i}}{i!}\left(  n_{1}-m+1\right)
_{i}\left(  n_{2}-m+1\right)  _{m-i}\left(  -v\right)  _{i}\left(  -u\right)
_{m-i}\nonumber\\
\widetilde{\psi}_{m}\left(  u,v\right)   &  :=\widetilde{\psi}_{m}^{\left(
1\right)  }\left(  u+v\right)  \widetilde{\psi}_{m}^{\left(  2\right)
}\left(  u,v\right) \nonumber
\end{align}
for $0\vee\left(  k-n_{3}\right)  \leq m\leq n_{1}\wedge n_{2}\wedge
k\wedge\left(  n_{1}+n_{2}-k\right)  $ (the number of these +1 is the
multiplicity of $1_{G_{\mathbf{n}}}$ in $\left[  N-k,k\right]  $).

\begin{proposition}
A basis for the $G_{\mathbf{n}}$-invariants is given by
\[
\psi_{m}\left(  x\right)  :=\sum_{u,v,u+v\leq k}\widetilde{\psi}_{m}\left(
u,v\right)  e_{u}\left(  x_{\ast}^{\left(  1\right)  }\right)  e_{v}\left(
x_{\ast}^{\left(  2\right)  }\right)  e_{k-u-v}\left(  x_{\ast}^{\left(
3\right)  }\right)  .
\]

\end{proposition}

This follows from the fact that $\widetilde{\psi}_{m}\left(  u,v\right)  $
satisfies the difference equation in Proposition \ref{diffceqn} (see
\cite[p.63, (3.11)]{DPJ}).

There are useful special values:%
\begin{align}
\widetilde{\psi}_{m}^{\left(  1\right)  }\left(  m\right)   &  =\left(
k-m\right)  !\left(  n_{1}+n_{2}-k-m+1\right)  _{k-m}\label{psim0}\\
\widetilde{\psi}_{m}^{\left(  2\right)  }\left(  u,0\right)   &  =\left(
n_{2}-m+1\right)  _{m}\left(  -u\right)  _{m}=\left(  -1\right)  ^{m}\left(
-n_{2}\right)  _{m}\left(  -u\right)  _{m}\nonumber\\
\widetilde{\psi}_{m}\left(  m,0\right)   &  =\left(  -1\right)  ^{k-m}\left(
m-n_{1}-n_{2}\right)  _{k-m}\left(  -n_{2}\right)  _{m}m!\left(  k-m\right)
!,\nonumber
\end{align}

\begin{lemma}
\label{uv<m}If $u+v<m$ then $\widetilde{\psi}_{m}^{\left(  2\right)  }\left(
u,v\right)  =0$.
\end{lemma}

\begin{proof}
The term for $i$ in $\widetilde{\psi}_{m}^{\left(  2\right)  }\left(
u,v\right)  $ is nonzero only if $i\leq v$ and $m-i\leq u$, that is, $m-u\leq
i\leq v$.
\end{proof}

An analogous structure is known for isotypes of $3$-part partitions, due to F.
Scarabotti \cite{Sc}. We are not pursuing this situation here because of the
complexity due to the added dimension and the numerous conditions on the parameters.

\section{\label{twocyc}Spherical functions at a $2$-cycle}

For an invariant polynomial $\psi\left(  x\right)  =\sum\limits_{u,v,u+v\leq
k}f\left(  u,v\right)  e_{u}\left(  x_{\ast}^{\left(  1\right)  }\right)
e_{v}\left(  x_{\ast}^{\left(  2\right)  }\right)  e_{k-u-v}\left(  x_{\ast
}^{\left(  3\right)  }\right)  $ we will determine $\rho\psi\left(
xg_{2}\right)  $ where $g_{2}=\left(  x_{1}^{\left(  1\right)  }%
,x_{1}^{\left(  2\right)  }\right)  $. In this section we will show that $\rho
g_{2}\psi_{m}=\frac{1}{n_{1}n_{2}}\left\{  \left(  m-n_{1}\right)  \left(
m-n_{2}\right)  -m\right\}  \psi_{m}$ for each $m$. This coefficient is then
used in Proposition \ref{Bsum}. Compute term-by-term. Let
\begin{align*}
p  &  =e_{u}\left(  x_{\ast}^{\left(  1\right)  }\right)  e_{v}\left(
x_{\ast}^{\left(  2\right)  }\right)  e_{k-u-v}\left(  x_{\ast}^{\left(
3\right)  }\right) \\
&  =\left\{  x_{1}^{\left(  1\right)  }e_{u-1}\left(  x_{>}^{\left(  1\right)
}\right)  +e_{u}\left(  x_{>}^{\left(  1\right)  }\right)  \right\}  \left\{
x_{1}^{\left(  2\right)  }e_{v-1}\left(  x_{>}^{\left(  2\right)  }\right)
+e_{v}\left(  x_{>}^{\left(  2\right)  }\right)  \right\}  e_{k-u-v}\left(
x_{\ast}^{\left(  3\right)  }\right) \\
g_{2}p  &  =x_{1}^{\left(  1\right)  }e_{u-1}\left(  x_{>}^{\left(  1\right)
}\right)  x_{1}^{\left(  2\right)  }e_{v-1}\left(  x_{>}^{\left(  2\right)
}\right)  e_{k-u-v}\left(  x_{\ast}^{\left(  3\right)  }\right)
+x_{1}^{\left(  1\right)  }e_{u}\left(  x_{>}^{\left(  1\right)  }\right)
e_{v-1}\left(  x_{>}^{\left(  2\right)  }\right)  e_{k-u-v}\left(  x_{\ast
}^{\left(  3\right)  }\right) \\
&  +e_{u-1}\left(  x_{>}^{\left(  1\right)  }\right)  x_{1}^{\left(  2\right)
}e_{v}\left(  x_{>}^{\left(  2\right)  }\right)  e_{k-u-v}\left(  x_{\ast
}^{\left(  3\right)  }\right)  +e_{u}\left(  x_{>}^{\left(  1\right)
}\right)  e_{v}\left(  x_{>}^{\left(  2\right)  }\right)  e_{k-u-v}\left(
x_{\ast}^{\left(  3\right)  }\right)  .
\end{align*}
Apply $\rho$ and use Lemma \ref{symme}%
\[
\rho g_{2}p=\frac{1}{n_{1}n_{2}}\left\{
\begin{array}
[c]{c}%
\left(  \left(  n_{1}-u\right)  \left(  n_{2}-v\right)  +uv\right)
e_{u}\left(  x_{\ast}^{\left(  1\right)  }\right)  e_{v}\left(  x_{\ast
}^{\left(  2\right)  }\right)  e_{k-u-v}\left(  x_{\ast}^{\left(  3\right)
}\right) \\
+\left(  u+1\right)  \left(  n_{2}-v+1\right)  e_{u+1}\left(  x_{\ast
}^{\left(  1\right)  }\right)  e_{v-1}\left(  x_{\ast}^{\left(  2\right)
}\right)  e_{k-u-v}\left(  x_{\ast}^{\left(  3\right)  }\right) \\
+\left(  n_{1}-u+1\right)  \left(  v+1\right)  e_{u-1}\left(  x_{\ast
}^{\left(  1\right)  }\right)  e_{v+1}\left(  x_{\ast}^{\left(  2\right)
}\right)  e_{k-u-v}\left(  x_{\ast}^{\left(  3\right)  }\right)
\end{array}
\right\}  .
\]
Thus
\begin{align*}
\rho g_{2}\psi &  =\sum\limits_{u,v,u+v\leq k}\left\{
\begin{array}
[c]{c}%
\left(  \left(  n_{1}-u\right)  \left(  n_{2}-v\right)  +uv\right)  f\left(
u,v\right) \\
+u\left(  n_{2}-v\right)  f\left(  u-1,v+1\right)  +\left(  n_{1}-u\right)
vf\left(  u+1,v-1\right)
\end{array}
\right\} \\
&  \times e_{u}\left(  x_{\ast}^{\left(  1\right)  }\right)  e_{v}\left(
x_{\ast}^{\left(  2\right)  }\right)  e_{k-u-v}\left(  x_{\ast}^{\left(
3\right)  }\right)  .
\end{align*}
Observe that values like $f\left(  -1,v+1\right)  $ or $f\left(
n_{1}+1,v\right)  $ do not appear. Let $f\left(  u,v\right)  =\widetilde{\psi
}_{m}^{\left(  1\right)  }\left(  u+v\right)  \widetilde{\psi}_{m}^{\left(
2\right)  }\left(  u,v\right)  $ from (\ref{psim1}), (\ref{psim2}). By Lemma
\ref{uv<m} $u+v<m$ implies $\widetilde{\psi}_{m}^{\left(  2\right)  }\left(
u,v\right)  =0$.

Let $C\left(  u,v,i\right)  $ denote the $i$-term in the sum (\ref{psim2}) for
$\widetilde{\psi}_{m}^{\left(  2\right)  }\left(  u,v\right)  $. We will
express $u\left(  n_{2}-v\right)  C\left(  u-1,v+1\right)  +\left(
n_{1}-u\right)  vC\left(  u+1,v-1,i\right)  $ in terms of $C\left(
u,v,i-1\right)  ,C\left(  u,v,i+1\right)  $ and $C\left(  u,v,i\right)  $. It
is more readable to display ratios like $C\left(  u+1,v-1,i\right)  /C\left(
u,v,i\right)  $ (resulting from straightforward calculations)%
\begin{align*}
a_{1}\left(  i\right)   &  :=u\left(  n_{2}-v\right)  \frac{C\left(
u-1,v+1,i\right)  }{C\left(  u,v,i\right)  }=\frac{\left(  n_{2}-v\right)
\left(  v+1\right)  \left(  m-i-u\right)  }{i-v-1}\\
a_{2}\left(  i\right)   &  :=\left(  n_{1}-u\right)  v\frac{C\left(
u+1,v-1,i\right)  }{C\left(  u,v,i\right)  }=\frac{\left(  n_{1}-u\right)
\left(  i-v\right)  \left(  u+1\right)  }{m-u-i-1}\\
b_{1}\left(  i\right)   &  :=-\left(  m+1-i\right)  \left(  n_{1}-m+i\right)
\frac{C\left(  u,v,i-1\right)  }{C\left(  u,v,i\right)  }=\frac{i\left(
m-i-u\right)  \left(  n_{2}+1-i\right)  }{i-v-1}\\
b_{2}\left(  i\right)   &  :=-\left(  i+1\right)  \left(  n_{2}-i\right)
\frac{C\left(  u,v,i+1\right)  }{C\left(  u,v,i\right)  }=\frac{\left(
-m+i\right)  \left(  i-v\right)  \left(  n_{1}-m+1+i\right)  }{1+i+u-m}.
\end{align*}
Then $a_{1}\left(  i\right)  -b_{1}\left(  i\right)  =\left(  i-v+n_{2}%
\right)  \left(  m-i-u\right)  ,~a_{2}\left(  i\right)  -b_{2}\left(
i\right)  =\left(  i-m+n_{1}-u\right)  \left(  v-i\right)  $. Thus%
\begin{align*}
&  u\left(  n_{2}-v\right)  C\left(  u-1,v+1,i\right)  +\left(  n_{1}%
-u\right)  vC\left(  u+1,v-1,i\right)  \\
&  =-\left(  m+1-i\right)  \left(  n_{1}-m+i\right)  C\left(  u,v,i-1\right)
-\left(  i+1\right)  \left(  n_{2}-i\right)  C\left(  u,v,i+1\right)  \\
&  +\left(  a_{1}\left(  i\right)  -b_{1}\left(  i\right)  +a_{2}\left(
i\right)  -b_{2}\left(  i\right)  \right)  C\left(  u,v,i\right)  .
\end{align*}
Apply $\sum_{i=0}^{m}$ to each line: the first line gives $u\left(
n_{2}-v\right)  \widetilde{\psi}_{m}^{\left(  2\right)  }\left(
u-1,v+1\right)  +\left(  n_{1}-u\right)  v\widetilde{\psi}_{m}^{\left(
2\right)  }\left(  u+1,v-1\right)  $, the second and third yield%
\begin{align*}
&  -\sum_{i=1}^{m}\left(  m+1-i\right)  \left(  n_{1}-m+i\right)  C\left(
u,v,i-1\right)  -\sum_{i=0}^{m-1}\left(  i+1\right)  \left(  n_{2}-i\right)
C\left(  u,v,i+1\right)  \\
&  +\sum_{i=0}^{m}\left(  a_{1}\left(  i\right)  -b_{1}\left(  i\right)
+a_{2}\left(  i\right)  -b_{2}\left(  i\right)  \right)  C\left(
u,v,i\right)  \\
&  =\sum_{i=0}^{m}\left\{  \left(  i-m\right)  \left(  n_{1}-m+1-i\right)
-i\left(  n_{2}-i\right)  +a_{1}\left(  i\right)  -b_{1}\left(  i\right)
+a_{2}\left(  i\right)  -b_{2}\left(  i\right)  \right\}  C\left(
u,v,i\right)  \\
&  =\sum_{i=0}^{m}\left\{  m\left(  m-n_{1}-n_{2}-1\right)  +n_{2}%
u+n_{1}v-2uv\right\}  C\left(  u,v,i\right)  ,
\end{align*}
a multiple of $\widetilde{\psi}_{m}^{\left(  2\right)  }\left(  u,v\right)  $.
The changes of summation variable are valid since $b_{1}\left(  0\right)
=0=b_{2}\left(  m\right)  $. Now add $\left(  \left(  n_{1}-u\right)  \left(
n_{2}-v\right)  +uv\right)  \widetilde{\psi}_{m}^{\left(  2\right)  }\left(
u,v\right)  $ to both sides and obtain

$\rho g_{2}\psi_{m}=\frac{1}{n_{1}n_{2}}\left\{  \left(  m-n_{1}\right)
\left(  m-n_{2}\right)  -m\right\}  \psi_{m}$. Let $m_{L}:=0\vee\left(
k-n_{3}\right)  $ and $m_{U}:=n_{1}\wedge n_{2}\wedge k\wedge\left(
n_{1}+n_{2}-k\right)  $. Thus the spherical function%
\[
\Phi^{\left[  N-k,k\right]  }\left(  g_{2}\right)  =\frac{1}{n_{1}n_{2}}%
\sum_{m=m_{L}}^{m_{U}}\left\{  \left(  m-n_{1}\right)  \left(  m-n_{2}\right)
-m\right\}  ,
\]
If $k\leq n_{1}\wedge n_{2}\wedge n_{3}$ then $m_{L}=0,m_{U}=k$ and
\[
\Phi^{\left[  N-k,k\right]  }\left(  g_{2}\right)  =\frac{k+1}{n_{1}n_{2}%
}\left(  n_{1}n_{2}-\frac{1}{2}k\left(  n_{1}+n_{2}\right)  +\frac{1}%
{3}k\left(  k-1\right)  \right)  .
\]
More generally let $\mu=m_{U}-m_{L}$ then
\begin{equation}
\Phi^{\left[  N-k,k\right]  }\left(  g_{2}\right)  =\frac{\mu+1}{n_{1}n_{2}%
}\left(  m_{L}^{2}+n_{1}n_{2}-\left(  m_{L}+\frac{\mu}{2}\right)  \left(
n_{1}+n_{2}\right)  +\left(  m_{L}+\frac{\mu}{3}\right)  \left(  \mu-1\right)
\right)  . \label{cycle2}%
\end{equation}

The corresponding situations for the $2$-cycles $\left(  x_{1}^{\left(
2\right)  },x_{1}^{\left(  3\right)  }\right)  $ and $\left(  x_{1}^{\left(
1\right)  },x_{1}^{\left(  3\right)  }\right)  $ are obtained by suitably
permuting the parameters $n_{1},n_{2},n_{3}$ in the formula. The multiplicity
of $1_{G_{\mathbf{n}}}$ in $V_{k}$ is symmetric in $\left\{  n_{i}\right\}  $,
namely\linebreak\ $\min\left\{  k,n_{1},\ldots,n_{1}+n_{2}-k,\ldots\right\}
+1$.

\subsection{Case $n_{1}+n_{2}=N$}

The same scheme can be used for $\mathbf{n}=\left(  n_{1},n_{2}\right)  $
(with $N=n_{1}+n_{2}$): the multiplicity of $\left[  N-k,k\right]  $ is one
for $0\leq k\leq n_{1}\wedge n_{2}$, the unique invariant polynomial (with
$\sum_{i=1}^{N}\frac{\partial}{\partial x_{i}}\psi\left(  x\right)  =0$) is%
\begin{align*}
\psi\left(  x\right)   &  :=\sum_{u=0}^{k}f\left(  u\right)  e_{u}\left(
x_{\ast}^{\left(  1\right)  }\right)  e_{k-u}\left(  x_{\ast}^{\left(
2\right)  }\right) \\
f\left(  u\right)   &  :=\left(  -1\right)  ^{u}\left(  n_{2}-k+1\right)
_{u}\left(  n_{1}-k+1\right)  _{k-u}%
\end{align*}
and with a similar calculation to the previous one%
\[
\rho g_{2}\psi=\sum_{u=0}^{k}\left\{
\begin{array}
[c]{c}%
\left(  \left(  n_{1}-u\right)  \left(  n_{2}-v\right)  +uv\right)  f\left(
u\right) \\
+u\left(  n_{2}-v\right)  f\left(  u-1\right)  +v\left(  n_{1}-u\right)
f\left(  u+1\right)
\end{array}
\right\}  e_{u}\left(  x_{\ast}^{\left(  1\right)  }\right)  e_{k-u}\left(
x_{\ast}^{\left(  2\right)  }\right)
\]
with $v=k-u$. We find%
\[
u\left(  n_{2}-v\right)  \frac{f\left(  u-1\right)  }{f\left(  u\right)
}+v\left(  n_{1}-u\right)  \frac{f\left(  u+1\right)  }{f\left(  u\right)
}=-u\left(  n_{1}+1-u\right)  -\left(  k-u\right)  \left(  n_{2}-k+1+u\right)
,
\]
and adding $\left(  n_{1}-u\right)  \left(  n_{2}-k-u\right)  +u\left(
k-u\right)  $ to both sides we obtain%
\begin{align*}
\rho g_{2}\psi &  =\left(  n_{1}n_{2}-\left(  n_{1}+n_{2}\right)
k+k^{2}-k\right)  \sum_{u=0}^{k}f\left(  u\right)  e_{u}\left(  x_{\ast
}^{\left(  1\right)  }\right)  e_{k-u}\left(  x_{\ast}^{\left(  2\right)
}\right) \\
&  =\left(  n_{1}n_{2}-\left(  n_{1}+n_{2}\right)  k+k^{2}-k\right)  \psi.
\end{align*}
Thus the spherical function%
\[
\Phi^{\left[  N-k,k\right]  }\left(  g_{2}\right)  =\frac{1}{n_{1}n_{2}%
}\left(  n_{1}n_{2}-\left(  n_{1}+n_{2}\right)  k+k^{2}-k\right)  .
\]

\section{\label{threecyc}Spherical functions at a $3$-cycle}

We use the $3$ -cycle $g_{3}=\left(  x_{1}^{\left(  1\right)  },x_{1}^{\left(
2\right)  },x_{1}^{\left(  3\right)  }\right)  $. We will determine $\rho
\psi\left(  xg_{3}\right)  $ for an invariant polynomial%
\begin{align*}
\psi\left(  x\right)   &  =\sum\limits_{u,v,u+v\leq k}f\left(  u,v\right)
\left\{  x_{1}^{\left(  1\right)  }e_{u-1}\left(  x_{>}^{\left(  1\right)
}\right)  +e_{u}\left(  x_{>}^{\left(  1\right)  }\right)  \right\}  \left\{
x_{1}^{\left(  2\right)  }e_{v-1}\left(  x_{>}^{\left(  2\right)  }\right)
+e_{v}\left(  x_{>}^{\left(  2\right)  }\right)  \right\}  \\
&  \times\left\{  x_{1}^{\left(  3\right)  }e_{k-u-v-1}\left(  x_{>}^{\left(
3\right)  }\right)  +e_{k-u-v}\left(  x_{>}^{\left(  3\right)  }\right)
\right\}
\end{align*}
The computation is quite a bit more involved than the $2$-cycle case. Apply
$g_{3}$ to the $\left(  u,v\right)  $-term; there are $8$ terms in the
expansion (and abbreviate $k-u-v=w$)%
\begin{align*}
&  x_{1}^{\left(  2\right)  }e_{u-1}\left(  x_{>}^{\left(  1\right)  }\right)
x_{1}^{\left(  3\right)  }e_{v-1}\left(  x_{>}^{\left(  2\right)  }\right)
x_{1}^{\left(  1\right)  }e_{w-1}\left(  x_{>}^{\left(  3\right)  }\right)
+x_{1}^{\left(  2\right)  }e_{u-1}\left(  x_{>}^{\left(  1\right)  }\right)
x_{1}^{\left(  3\right)  }e_{v-1}\left(  x_{>}^{\left(  2\right)  }\right)
e_{w}\left(  x_{>}^{\left(  3\right)  }\right)  \\
&  +x_{1}^{\left(  2\right)  }e_{u-1}\left(  x_{>}^{\left(  1\right)
}\right)  e_{v}\left(  x_{>}^{\left(  2\right)  }\right)  x_{1}^{\left(
1\right)  }e_{w-1}\left(  x_{>}^{\left(  3\right)  }\right)  +x_{1}^{\left(
2\right)  }e_{u-1}\left(  x_{>}^{\left(  1\right)  }\right)  e_{v}\left(
x_{>}^{\left(  2\right)  }\right)  e_{w}\left(  x_{>}^{\left(  3\right)
}\right)  \\
&  +e_{u}\left(  x_{>}^{\left(  1\right)  }\right)  x_{1}^{\left(  3\right)
}e_{v-1}\left(  x_{>}^{\left(  2\right)  }\right)  x_{1}^{\left(  1\right)
}e_{w-1}\left(  x_{>}^{\left(  3\right)  }\right)  +e_{u}\left(
x_{>}^{\left(  1\right)  }\right)  x_{1}^{\left(  3\right)  }e_{v-1}\left(
x_{>}^{\left(  2\right)  }\right)  e_{w}\left(  x_{>}^{\left(  3\right)
}\right)  \\
&  +e_{u}\left(  x_{>}^{\left(  1\right)  }\right)  e_{v}\left(
x_{>}^{\left(  2\right)  }\right)  x_{1}^{\left(  1\right)  }e_{w-1}\left(
x_{>}^{\left(  3\right)  }\right)  +e_{u}\left(  x_{>}^{\left(  1\right)
}\right)  e_{v}\left(  x_{>}^{\left(  2\right)  }\right)  e_{w}\left(
x_{>}^{\left(  3\right)  }\right)  .
\end{align*}
Symmetrize each term using Lemma \ref{symme}, and denote%
\[
P\left(  u,v,w\right)  :=\frac{1}{n_{1}n_{2}n_{3}}e_{u}\left(  x_{\ast
}^{\left(  1\right)  }\right)  e_{v}\left(  x_{\ast}^{\left(  2\right)
}\right)  e_{w}\left(  x_{\ast}^{\left(  3\right)  }\right)
\]
:
\begin{align*}
&  uvwP\left(  u,v,w\right)  +\left(  n_{1}-u+1\right)  v\left(  w+1\right)
P\left(  u-1,v,w+1\right)  \\
&  +u\left(  v+1\right)  \left(  n_{3}-w+1\right)  P\left(  u,v+1,w-1\right)
+\left(  n_{1}-u+1\right)  \left(  v+1\right)  \left(  n_{3}-w\right)
P\left(  u-1,v+1,w\right)  \\
&  +\left(  u+1\right)  \left(  n_{2}-v+1\right)  wP\left(  u+1,v-1,w\right)
+\left(  n_{1}-u\right)  \left(  n_{2}-v+1\right)  \left(  w+1\right)
P\left(  u,v-1,w+1\right)  \\
&  +\left(  u+1\right)  \left(  n_{2}-v\right)  \left(  n_{3}-w+1\right)
P\left(  u+1,v,w-1\right)  +\left(  n_{1}-u\right)  (n_{2}-v)\left(
n_{3}-w\right)  P(u,v,w)
\end{align*}
respectively. Changing indices as appropriate we obtain%
\begin{align*}
\rho\psi\left(  xg_{3}\right)   &  =\sum\limits_{u,v,u+v\leq k}P\left(
u,v,w\right)  \times\\
&  \left\{
\begin{array}
[c]{c}%
uvwf\left(  u,v\right)  +\left(  n_{1}-u\right)  vwf\left(  u+1,v\right)  \\
+uv\left(  n_{3}-w\right)  f\left(  u,v-1\right)  +\left(  n_{1}-u\right)
v\left(  n_{3}-w\right)  f\left(  u+1,v-1\right)  \\
+u\left(  n_{2}-v\right)  wf\left(  u-1,v+1\right)  +\left(  n_{1}-u\right)
\left(  n_{2}-v\right)  wf\left(  u,v+1\right)  \\
+u\left(  n_{2}-v\right)  \left(  n_{3}-w\right)  f\left(  u-1,v\right)
+\left(  n_{1}-u\right)  \left(  n_{2}-v\right)  \left(  n_{3}-w\right)
f\left(  u,v\right)
\end{array}
\right\}  ..
\end{align*}
Mow set $f\left(  u,v\right)  =\widetilde{\psi}_{m}\left(  u,v\right)  $ and
determine the coefficient $c_{m}$ in $\rho\psi_{m}\left(  xg_{3}\right)
=\sum_{n}c_{n}\psi_{n}\left(  x\right)  $ (equivalently the expression in
$\left\{  \cdot\right\}  $ denoted $S_{m}\left(  u,v\right)  $ equals
$\sum_{n}c_{n}\widetilde{\psi}_{n}\left(  u,v\right)  $).

In $\psi_{m}^{\left(  1\right)  }\left(  m\right)  $ only the $j=k-m$ term is
nonzero, and in $\widetilde{\psi}_{m}^{\left(  2\right)  }\left(  u,0\right)
$ only the $i=0$ term is nonzero. Furthermore $u<m$ implies $\widetilde{\psi
}_{m}\left(  u,0\right)  =0,$ (because of the factor $\left(  -u\right)  _{m}$
in $\widetilde{\psi}_{m}^{\left(  2\right)  }\left(  u,v\right)  $). By Lemma
\ref{uv<m} $u+v<n$ implies $\widetilde{\psi}_{n}\left(  u,v\right)  =0$. Thus
$S_{m}\left(  u,0\right)  =0$ for $u<m-1$ (note the terms $vf\left(
\ast,v-1\right)  =0$). The following is used to determine the coefficients
needed for Proposition \ref{Bsum}.

\begin{proposition}
\label{getcm}Suppose $f=\sum_{n=0}^{k}c_{n}\widetilde{\psi}_{n}$ and $f\left(
u,0\right)  =0$ for $u<m-1$ then%
\begin{equation}
c_{m}=\frac{1}{\widetilde{\psi}_{m}\left(  m,0\right)  }\left\{  f\left(
m.0\right)  -\frac{m\left(  k-m-n_{3}\right)  }{n_{1}+n_{2}-2m+2}f\left(
m-1,0\right)  \right\}  . \label{cmf}%
\end{equation}

\end{proposition}

\begin{proof}
The coefficients $c_{j}=0$ for $j<m-1$. Indeed let $i:=\min\left\{
j:c_{j}\neq0\right\}  $ then $f\left(  i,0\right)  =c_{i}\psi_{i}\left(
i,0\right)  \neq0$, and thus $i\geq m-1$. It remains to show that
$\widetilde{\psi}_{m-1}\left(  m.0\right)  -\frac{m\left(  k-m-n_{3}\right)
}{n_{1}+n_{2}-2m+2}\widetilde{\psi}_{m-1}\left(  m-1,0\right)  =0$. From
$\widetilde{\psi}_{m-1}^{\left(  2\right)  }\left(  u,0\right)  \allowbreak
=\left(  n_{2}-m+2\right)  _{m-1}\left(  -u\right)  _{m-1}$ we find
$\widetilde{\psi}_{m-1}^{\left(  2\right)  }\left(  m,0\right)  /\widetilde
{\psi}_{m-1}^{\left(  2\right)  }\left(  m-1,0\right)  =m$. In the sum for
$\widetilde{\psi}_{m-1}^{\left(  1\right)  }\left(  u\right)  $ at $u=m-1$
only the $j=k-m+1$ term appears and at $u=m$ only the $j=k-m$ term appears.
Using this fact we find%
\begin{align*}
\frac{\widetilde{\psi}_{m-1}^{\left(  1\right)  }\left(  m\right)
}{\widetilde{\psi}_{m-1}^{\left(  1\right)  }\left(  m-1\right)  }  &
=\frac{k-m-n_{3}}{n_{1}+n_{2}-2m+2},\\
\frac{\widetilde{\psi}_{m-1}\left(  m,0\right)  }{\widetilde{\psi}%
_{m-1}\left(  m-1,0\right)  }  &  =\frac{m\left(  k-m-n_{3}\right)  }%
{n_{1}+n_{2}-2m+2},
\end{align*}
and this concludes the proof.
\end{proof}

We need to evaluate $\widetilde{\psi}_{m}\left(  u-1,1\right)  ,\widetilde
{\psi}_{m}\left(  u,1\right)  ,\widetilde{\psi}_{m}\left(  u-1,0\right)
,\widetilde{\psi}_{m}\left(  u,0\right)  $ at $u=m-1,m$. The first and third
of these vanish at $u=m-1$, by Lemma \ref{uv<m}. Besides the values in
formulas (\ref{psim0}) the following are needed:%
\[
\widetilde{\psi}_{m}^{\left(  1\right)  }\left(  m+1\right)  =-\left(
k-m\right)  !\left(  n_{1}+n_{2}-k-m+1\right)  _{k-m-1}\left(  n_{3}%
-k+m+1\right)  ;
\]

\begin{align*}
\widetilde{\psi}_{m}^{\left(  2\right)  }\left(  m,0\right)   &  =m!\left(
-n_{2}\right)  _{m}\\
\widetilde{\psi}_{m}^{\left(  2\right)  }\left(  m-1,1\right)   &  =m!\left(
n_{1}-m+1\right)  \left(  1-n_{2}\right)  _{m-1}\\
\widetilde{\psi}_{m}^{\left(  2\right)  }\left(  m,1\right)   &  =\left(
-1\right)  ^{m}m!\left(  n_{2}-m+1\right)  _{m-1}\left(  n_{2}-m\left(
n_{1}-m+1\right)  \right)  .
\end{align*}
To organize the calculations let%
\begin{align*}
A &  :=\widetilde{\psi}_{m}^{\left(  1\right)  }\left(  m+1\right)
/\widetilde{\psi}_{m}^{\left(  1\right)  }\left(  m\right)  =-\frac
{n_{3}-k+m+1}{n_{1}+n_{2}-2m}\\
B_{1} &  :=\widetilde{\psi}_{m}^{\left(  2\right)  }\left(  m-1,1\right)
/\widetilde{\psi}_{m}^{\left(  2\right)  }\left(  m,0\right)  =-\frac
{n_{1}-m+1}{n_{2}}\\
B_{2} &  :=\widetilde{\psi}_{m}^{\left(  2\right)  }\left(  m,1\right)
/\widetilde{\psi}_{m}^{\left(  2\right)  }\left(  m,0\right)  =\frac
{n_{2}-m\left(  n_{1}-m+1\right)  }{n_{2}}%
\end{align*}%
\begin{align*}
\mathcal{T}f\left(  u,v\right)   &  :=\frac{1}{\widetilde{\psi}_{m}^{\left(
1\right)  }\left(  m\right)  \widetilde{\psi}_{m}^{\left(  2\right)  }\left(
m,0\right)  }\left\{  f\left(  u,0\right)  -C_{m}f\left(  u-1,0\right)
\right\}  \\
C_{m} &  :=\frac{m\left(  k-m-n_{3}\right)  }{n_{1}+n_{2}-2m+2}.
\end{align*}
Three of the $\mathcal{T}$- evaluations are nonzero:%
\begin{align*}
\mathcal{T}\left\{  u\left(  n_{2}-v\right)  wf\left(  u-1,v+1\right)
\right\}   &  =mn_{2}\left(  k-m\right)  B_{1},\\
\mathcal{T}\left\{  \left(  n_{1}-u\right)  \left(  n_{2}-v\right)  wf\left(
u,v+1\right)  \right\}   &  =\left(  n_{1}-m\right)  n_{2}\left(  k-m\right)
AB_{2}\\
&  -\left(  n_{1}-m+1\right)  n_{2}\left(  k-m+1\right)  CB_{1},\\
\mathcal{T}\left\{  \left(  n_{1}-u\right)  \left(  n_{2}-v\right)  \left(
n_{3}-w\right)  f\left(  u,v\right)  \right\}   &  =\left(  n_{1}-m\right)
n_{2}\left(  n_{3}-k+m\right)  .
\end{align*}
The omitted cases are due to $\widetilde{\psi}_{m}^{\left(  2\right)  }\left(
m-1,0\right)  =0=\widetilde{\psi}_{m}^{\left(  2\right)  }\left(
m-2,0\right)  $. Let%
\[
\xi\left(  k,m\right)  :=\left(  m+1\right)  \left(  n_{3}-k+m+1\right)
(k-m)\frac{n_{1}n_{2}-m^{2}}{n_{1}+n_{2}-2m}%
\]
then%
\begin{align*}
\left(  n_{1}-m\right)  n_{2}\left(  k-m\right)  AB_{2} &  =\left(
n_{1}+1\right)  m\left(  k-m\right)  \left(  n_{3}-k+m+1\right)  -\xi\left(
k,m\right)  \\
-\left(  n_{1}-m+1\right)  n_{2}\left(  k-m+1\right)  CB_{1} &  =-n_{1}%
m\left(  k-m+1\right)  \left(  n_{3}-k-m\right)  +\xi\left(  k,m-1\right)  .
\end{align*}
We must deal with the exceptional case $2m=n_{1}+n_{2}$: if $k<\left(
n_{1}\wedge n_{2}\right)  $ then $m\leq k$ and $2m<n_{1}+n_{2},$ so suppose
$n_{1}\wedge n_{2}\leq k$ and $m=n_{1}\wedge n_{2}$, and $2m=n_{1}+n_{2}$
implies $m=n_{1}=n_{2}$ so that the term $AB_{2}$ does not occur. In fact
there is more detail: if $k>n_{1}=n_{2}$ then $n_{1}+n_{2}-k<n_{1}$ giving the
bound $m\leq n_{1}+n_{2}-k$ and $n_{1}+n_{2}-2m\geq k-m>0$, else if $k=m$ then
$\widetilde{\psi}_{1}=1$ and $A=1$.

Adding the six terms we have shown that the coefficient $c_{m}$ in the
expansion $\rho g_{3}\psi_{m}=\sum\limits_{n}c_{n}\psi_{n}$ is%
\begin{align*}
c_{m} &  =\frac{1}{n_{1}n_{2}n_{3}}\left\{  \zeta\left(  k,m\right)
-\xi\left(  k,m\right)  +\xi\left(  k,m-1\right)  \right\}  \\
\zeta\left(  k,m\right)   &  :=m^{2}\left(  3k-2m\right)  +m\left(  n_{3}%
^{2}-k^{2}\right)  -\left(  n_{3}-k+m\right)  \left(  m\left(  n_{1}%
+n_{2}+n_{3}\right)  -n_{1}n_{2}\right)  .
\end{align*}
Thus the spherical function $\Phi^{\left[  N-k,k\right]  }\left(
g_{3}\right)  =\frac{1}{n_{1}n_{2}n_{3}}\sum_{m=m_{L}}^{m_{U}}\zeta\left(
k,m\right)  -\xi\left(  k,m_{U}\right)  +\xi\left(  k,m_{L}-1\right)  $, by
telescoping. Recall $m_{L}:=0\vee\left(  k-n_{3}\right)  $ and $m_{U}%
:=n_{1}\wedge n_{2}\wedge k\wedge\left(  n_{1}+n_{2}-k\right)  $. By
definition $\xi\left(  k,-1\right)  =0=\xi\left(  k,k-n_{3}-1\right)  $ so
that $\xi\left(  k,m_{L}-1\right)  =0$ , also $\xi\left(  k,k\right)  =0$. The
nonzero values of $\xi\left(  k,m_{U}\right)  $ are
\begin{align*}
\xi\left(  k,m_{U}\right)   &  =\left(  m_{U}+1\right)  \left(  m_{U}-\left(
k-n_{3}\right)  +1\right)  m_{U}\left(  k-m_{U}\right)  ,~m_{U}=n_{1}\wedge
n_{2},\\
\xi\left(  k,m_{U}\right)   &  =\left(  m_{U}+1\right)  \left(  m_{U}-\left(
k-n_{3}\right)  +1\right)  \left(  n_{1}n_{2}-m_{U}^{2}\right)  ,~m_{U}%
=n_{1}+n_{2}-k.
\end{align*}
One of the first two factors is $\left(  m_{U}-m_{L}+1\right)  $ since
$m_{L}=0\vee k-n_{3}$.$.$We point out that the value of a spherical function
times $n_{1}n_{2}n_{3}$ is an integer, because the character table of
$\mathcal{S}_{N}$ has all integer entries, and necessarily the denominator in
$\xi\left(  k,\mu\right)  $ cancels out. If $k\leq n_{3}$ ($m_{L}=0$) let
$\mu=m_{U}$ and then%
\[
\Phi^{\left[  N-k,k\right]  }\left(  g_{3}\right)  =\frac{\mu+1}{n_{1}%
n_{2}n_{3}}\left\{
\begin{array}
[c]{c}%
n_{1}n_{2}n_{3}-\frac{1}{2}\mu\left(  n_{1}n_{2}+n_{1}n_{3}+n_{1}n_{3}\right)
+\frac{1}{6}N\mu\left(  \mu-1\right)  \\
+\frac{1}{2}\left(  k-\mu\right)  \left(  \mu\left(  N-k+\mu+1\right)
-2n_{1}n_{2}\right)  -\frac{1}{\mu+1}\xi\left(  k,\mu\right)
\end{array}
\right\}  .
\]
The simplest case is $\mu=k,$ $\xi\left(  k,k\right)  =0$. The sum can be
explicitly found in general but tends to be complicated. Here is one way to
display the sum (with $\mu=m_{U}-m_{L}$, $\nu=m_{L}$, $\delta=k-m_{U}$),
omitting the factor $\dfrac{\mu+1}{n_{1}n_{2}n_{3}}$
\begin{align}
&  \frac{1}{6}N\mu\left(  \mu-1\right)  +\frac{1}{2}N\left(  \mu\nu+\mu
\delta+2\nu\delta\right)  -\frac{1}{2}\left(  \mu+2\nu\right)  \left(
n_{1}n_{2}+n_{1}n_{3}+n_{1}n_{3}\right)  +n_{1}n_{2}n_{3}\label{sumPhi}\\
&  +\frac{1}{2}\mu\nu\left(  \nu-1\right)  +\left(  \nu-\delta\right)  \left(
n_{1}n_{2}+\nu\delta\right)  -\frac{1}{2}\mu\delta\left(  \delta-1\right)
-\frac{1}{\mu+1}\xi\left(  k,m_{U}\right)  .\nonumber
\end{align}

\subsection{\label{mult1}Special situations}

Multiplicity equal to one arises when $m_{L}=m_{U}$ at (1) $k=n_{1}+n_{3}$,
(2) $k=n_{2}+n_{3}$, (3) $k=n_{1}+n_{2}$, (4) $k=\frac{N}{2}$. For case (1)
$m_{L}=k-n_{3}=m_{U}=n_{1}$, then $\mu=0,$ $\nu=n_{1}$, $\delta=n_{3}$ and by
formula (\ref{sumPhi}) $\Phi^{\left[  N-k,k\right]  }\left(  g_{3}\right)
=-\dfrac{1}{n_{2}}$. A similar calculation shows case (2) yields $-\dfrac
{1}{n_{1}}$. For case (3) $m_{L}=0=n_{1}+n_{2}-k=m_{U}$, then $\mu
=0,~\nu=0,~\delta=n_{1}+n_{2}$ and by the same formula $\Phi^{\left[
N-k,k\right]  }\left(  g_{3}\right)  =-\dfrac{1}{n_{3}}$. In these cases there
are implicit bounds such as $n_{2}\geq\frac{N}{2},n_{1}+n_{3}\leq\frac{N}{2}$,
following from $2k\leq N$. Applying these parameters for the $2$-cycle case
when $g_{2}=\left(  x_{1}^{\left(  2\right)  },x_{1}^{\left(  1\right)
}\right)  $ and using formula (\ref{cycle2}) for $\Phi^{\left[  N-k,k\right]
}\left(  g_{2}\right)  $ one obtains (1) $-\dfrac{1}{n_{2}}$, (2) $-\dfrac
{1}{n_{1}}$, (3) $1$.

For case (4) (when $N$ is even) $m_{L}=\frac{N}{2}-n_{3}=m_{U}=n_{1}%
+n_{2}-\frac{N}{2}$, $\delta=n_{3}$. The resulting values can be written as%
\begin{align*}
\Phi^{\left[  N/2,N/2\right]  }\left(  g_{3}\right)    & =\frac{1}{n_{1}%
n_{2}n_{3}}\left\{  -\prod\limits_{i=1}^{3}\left(  \frac{N}{2}-n_{i}\right)
-\sum_{1\leq i<j\leq3}\left(  \frac{N}{2}-n_{i}\right)  \left(  \frac{N}%
{2}-n_{j}\right)  \right\}  ,\\
\Phi^{\left[  N/2,N/2\right]  }\left(  g_{2}\right)    & =\dfrac{1}{n_{1}%
n_{2}}\left(  \left(  \frac{N}{2}-n_{1}\right)  \left(  \frac{N}{2}%
-n_{2}\right)  -\frac{N}{2}+n_{3}\right)  .
\end{align*}

Another example is $n_{1}=n_{2}=n_{3}=n$ (and $N=3n$). If $k\leq n$ then
$\nu=m_{L}=0,m_{U}=k$, $\mu=k$, $\delta=0$ and%
\[
\Phi^{\left[  N-k,k\right]  }\left(  g_{3}\right)  =\frac{k+1}{n^{2}}\left(
n^{2}-\frac{3}{2}nk+\frac{1}{2}k\left(  k-1\right)  \right)  .
\]
If $n\leq k\leq\frac{3}{2}n$ then $\nu=m_{L}=k-n,m_{U}=2n-k$ (since
$n\geq2n-k$), $\mu=3n-2k$, $\delta=2k-2n$ and
\[
\Phi^{\left[  N-k,k\right]  }\left(  g_{3}\right)  =\frac{3n-2k+1}{n^{2}%
}\left(  n^{2}-\frac{3}{2}nk+\frac{1}{2}k\left(  k-1\right)  \right)  .
\]

\end{document}